\documentclass{amsart}

\usepackage{graphics} 

\usepackage[utf8]{inputenc}

\usepackage[usenames]{color}
\usepackage{hyperref}
\usepackage{fancyhdr}

\usepackage{epstopdf}
\usepackage{amsmath}
\usepackage{amssymb}
\usepackage{amsfonts}
\usepackage{amscd}

\usepackage{mathtools}

\usepackage{mathrsfs}
\usepackage{pgf,tikz}
\usetikzlibrary{arrows}
\usepackage{tikz-cd}  

\usepackage{enumitem}
\setlist{nolistsep}
\usepackage{layout}

\usepackage{ stmaryrd } 
\usepackage{cite}

\usepackage{epsdice} 

\newcommand{\salg}[1]{\mathfrak {#1}}

\renewcommand{\ker}{\operatorname{ker}}

\newcommand{\im}{\operatorname{im}}

\newcommand{\End}{\operatorname{End}}

\newcommand{\Hom}{\operatorname{Hom}}
\newcommand{\Ext}{\operatorname{Ext}}

\newcommand{\Ob}{\operatorname{Ob}}

\newcommand{\cat}[1]{\mathbf{#1}}
\newcommand{\Sets}{\cat{Sets}}

\newcommand{\Mod}{\cat{Mod}}

\newcommand{\MeasSets}{\cat{Meas}}

\newcommand{\supp}{\operatorname{supp}}

\def \d{\mbox{\(\,\mathrm{d}\)}}
 
\renewcommand{\epsilon}{\varepsilon}

\newcommand{\norm}[1]{\left\|#1\right\|}

\newcommand{\sm}{\setminus}

\newcommand{\Rr}{\mathbb{R}}

\newcommand{\Nn}{\mathbb{N}}

\newcommand{\Zz}{\mathbb{Z}}

\newcommand{\Ff}{\mathbb{F}}

\newcommand{\set}[2]{\{\,#1 \, : \, #2\,\} }
\newcommand{\bigset}[2]{\left\{\,#1 \, : \, #2\,\right\} }

\newcommand{\qbinom}[2]{ {#1 \brack #2}_q }
\newcommand{\fwbinom}[2]{ {#1 \brace #2} }

\newcommand{\Sring}{\mathcal{A}}
\newcommand{\Sinf}{\cat{S}}
\newcommand{\Smon}{\mathcal S}
\newcommand{\Us}{\mathbf{1}}

\newcommand{\Exp}{\operatorname{Exp}}

\newcommand{\sheaf}[1]{\mathcal{#1}}

\newcommand{\keyt}[1]{\emph{#1}}

\newtheorem{theorem}{Theorem}
\newtheorem{lem}{Lemma}
\newtheorem{prop}{Proposition}

\theoremstyle{definition}
\newtheorem{defi}{Definition}
\newtheorem{ex}[theorem]{Example}

\theoremstyle{remark}
\newtheorem{remark}{Remark}

\begin{document}

\title[Homological characterization of  multinomial coefficients]{A homological characterization of generalized multinomial coefficients related to the entropic chain rule}


\author{Juan Pablo Vigneaux}
\address{Institut de Math\'ematiques de Jussieu-Paris Rive Gauche (IMJ-PRG).
Universit\'e de Paris, Sorbonne Universit\'e \& CNRS. 
 F-75013, Paris, France.}
\curraddr{Max Planck Institute for Mathematics in the Sciences, Inselstraße 22, 04103 Leipzig, Germany.}
\email{vigneaux@mis.mpg.de}
\thanks{The work presented here was part of my doctoral dissertation. I want to thank my Ph.D. advisor, Prof. Daniel Bennequin, for his support.}



\date{\today}

\dedicatory{}

\begin{abstract}
There is an asymptotic relationship between the multiplicative relations among multinomial coefficients and the (additive) recurrence property of Shannon entropy known as the chain rule. We show that both types of identities are manifestations of a unique algebraic construction: a $1$-cocycle condition in \emph{information cohomology}, an algebraic invariant of phesheaves of modules on \emph{information structures} (categories of observables). Baudot and Bennequin introduced this cohomology and  proved that Shannon entropy represents the only nontrivial cohomology class in degree $1$  when the coefficients are a natural presheaf of probabilistic functionals. The author obtained later a $1$-parameter family of deformations of that presheaf, in such a way that each Tsallis $\alpha$-entropy appears as the unique $1$-cocycle associated to the parameter $\alpha$. In this article, we introduce a new presheaf of \emph{combinatorial functionals}, which are measurable functions of finite arrays of integers; these arrays represent \emph{histograms} associated to random experiments. In this case, the only cohomology class in degree $0$ is generated by the exponential function and $1$-cocycles are Fontené-Ward generalized multinomial coefficients.
As a byproduct, we get a simple combinatorial analogue of the fundamental equation of information theory that characterizes the generalized binomial coefficients. The asymptotic relationship mentioned above is extended to a correspondence between certain generalized multinomial coefficients and any $\alpha$-entropy, that sheds new light on the meaning of the chain rule and its deformations.
\end{abstract}

\maketitle
\tableofcontents

\section{Motivations and main results}\label{sec:intro}

It is well known that the multinomial coefficients\footnote{For integers  $n,k_1,...,k_s$ such that $\sum_{i=1}^s k_i = n$, one has $$ {n\choose k_1,...,k_s} := \frac{n!}{k_1! \cdots k_s!}.$$ This expression counts words $w\in \{a_1,...,a_s\}^n$ where the symbol $a_i$ appears $k_i$ times, for each $i$.} are asymptotically related to Shannon entropy: if $(p_1,...,p_s)$ is a probability vector and $n\in \Nn$, 
\begin{equation}
{n \choose p_1n,...,p_sn} := \frac{\Gamma(n+1)}{\Gamma(p_1n+1) \cdots \Gamma(p_sn+1)}= \exp(n S_1(p_1,...,p_s)+o(n))
\end{equation}
where $S_1$, or more precisely $S_1^{(s)}$, denotes Shannon entropy in nats,
\begin{equation}\label{eq:Shannon_entropy}
S_1(p_1,...,p_s) := -\sum_{i=1}^s p_i \ln p_i.
\end{equation}
The fact that the multiplicative relations between these coefficients 
 translate asymptotically into the entropic chain rule is however never mentioned. For instance, from the  \emph{combinatorial} identity
 \begin{equation}\label{intro:recurrence-multinomial-3}
 {n\choose p_1n,p_2n,p_3n} = {n \choose (p_1+p_2)n,p_3n}{(p_1+p_2)n \choose p_1n,p_2n},
 \end{equation}
 one  can  deduce---taking the logarithm of both sides, normalizing by $n$ and then letting $n\to \infty$---that
 \begin{equation}\label{eq_ex_entropic_chain_rule}
 S_1(p_1,p_2,p_3) = S_1(p_1+p_2, p_3) + (p_1+p_2)S_1\left(\frac{p_1}{p_1+p_2},\frac{p_2}{p_1+p_2}\right).
 \end{equation}
 Both equalities are  induced by a grouping of the arguments (``coarse graining''), which can be represented as a surjection from $\{1,2,3\}$ to $\{1,2\}$ that maps $1$ to $1$, $2$ to $1$, and $3$ to $2$. 
 
 The additive relations exemplified by \eqref{eq_ex_entropic_chain_rule}, known in information theory as the \emph{chain rule}, serve as a fundamental property to algebraically characterize the entropy. Let us denote by $\Delta^n$ the standard simplex $\set{(x_0,...,x_n)\in \Rr^n}{\sum_{i=0}^n x_i = 1}$, and $[n]$ the set $\{1,...,n\}$.  Shannon \cite{Shannon1948}  proved that $\{S_1^{(n+1)}:\Delta^n\to \Rr\}_{n\in \Nn}$ are the only continuous functions (up to a multiplicative constant) that vanish on the vertexes of each simplex $\Delta^n$, make $S_1^{(n)}(1/n,...,1/n)$ monotonic in $n$, and satisfy the chain rule induced by any surjection $[n]\to [m]$.  It is natural to ask if the multinomial coefficients can be algebraically characterized in an analogous way. 
 
 The equality \eqref{eq_ex_entropic_chain_rule}, together with the symmetry of the entropy, imply that $$s_1(x) := S^{(2)}_1(x,1-x)=-x \ln x -(1-x)\ln(1-x)$$ is a solution of the so-called \emph{fundamental equation of information theory (FEITH)}:
\begin{multline}\label{eq:FEITH}
\forall x,y\in [0,1) \text{ such that } x+y \leq 1, \quad\\ u(x) + (1-x)u \left(\frac{y}{1-x} \right) =  u(y) + (1-y)u \left(\frac{x}{1-y} \right).
\end{multline} 
 This functional equation was first introduced by Tverberg \cite{Tverberg1958}, who proved that every integrable and symmetric solution of \eqref{eq:FEITH} is a multiple of  $s_1(x)$. The regularity condition can be weakened to mere measurability \cite{Kannappan1973}. Tverberg's result gives an alternative algebraic characterization of Shannon entropy. Furthermore, the fundamental equation is also relevant in other areas of mathematics: it appears in Cathelineau's computations of the degree-one homology of $SL_2$ over a field of characteristic zero with coefficients in the adjoint action \cite{Cathelineau1988}, as well as in subsequent work by Elbaz-Vincent and Gangl \cite{Elbaz2002,Elbaz2015} and Bloch and Esnault \cite{Bloch2003} connected to polylogarithms and motives.  Kontsevich \cite{Kontsevich1995} used a version of the FEITH to introduce the entropy modulo $p$ and also gave a cohomological interpretation of this functional equation. 
 
 The algebraic characterizations of entropy already mentioned accept a $1$-parameter family of deformations. For any  $\alpha>0$, define  $S_\alpha \equiv S_\alpha^{(s)}:\Delta^{s-1}\to \Rr$ by the formula
 \begin{equation}
 S_\alpha(p_1,...,p_s) := \frac{1}{1-\alpha} \sum_{i=1}^s p_i^\alpha -1,
 \end{equation}
 in such a way that $S_\alpha \to S_1$ when $\alpha\to 1$. 
 This function was first introduced by Havrda-Charv\'{a}t \cite{Havrda1967} as \emph{structural $\alpha$-entropy} and nowadays it is mostly known as Tsallis $\alpha$-entropy. It satisfies a deformed chain rule where the weights in front of each term are raised to the power $\alpha$, e.g.
  \begin{equation}\label{eq_ex_entropic_chain_rule_deformed}
 S_\alpha(p_1,p_2,p_3) = S_\alpha(p_1+p_2, p_3) + (p_1+p_2)^\alpha S_\alpha\left(\frac{p_1}{p_1+p_2},\frac{p_2}{p_1+p_2}\right).
 \end{equation}
 In \cite{Havrda1967}, this property plays a fundamental role in an algebraic characterization of $S_\alpha$ (up  to a multiplicative constant) analogue to Shannon's characterization of $S_1$. Along the same line, Dar\'{o}czy introduced a generalized FEITH,
 \begin{multline}\label{eq:FEITH-def}
\forall x,y\in [0,1) \text{ such that } x+y \leq 1, \quad\\ u(x) + (1-x)^\alpha u \left(\frac{y}{1-x} \right) =  u(y) + (1-y)^\alpha u \left(\frac{x}{1-y} \right).
\end{multline} 
with boundary condition $u(0)=u(1)$, and proved that its only solutions are multiples of 
$$s_\alpha (x)= \frac{1}{1-\alpha} (x^\alpha + (1-x)^\alpha -1),$$
without any hypothesis on the regularity of $u$. 

 Up to this point, there is no general \emph{combinatorial} counterpart to the entropies $S_\alpha$ and their chain rule; the latter could be judged as a purely formal rule, without further implications. However, we showed in a previous article \cite{Vigneaux2019} that the Gaussian $q$-multinomial coefficients are asymptotically related to the $2$-entropy, giving a concrete combinatorial meaning to the deformed chain rule (for a precise statement, see the examples after Proposition \ref{prop:asymptotic-relation-n-cocycles}). Similar results may hold for other combinatorial quantities and other values of $\alpha$. 

At the algebraic level, there is more than an ``analogy'' between the multiplicative relations among multinomial coefficients and the entropic chain rule: we establish in this article that both are particular cases of a general construction called \emph{information cohomology}. This theory was first introduced by Baudot and Bennequin in \cite{Baudot2015} and further developed by the author in \cite{Vigneaux2019-thesis}. These works prove that each entropy $S_\alpha$, for $\alpha>0$, is the unique $1$-cocycle in information cohomology with coefficients in certain module of probabilistic functionals $\sheaf F_\alpha$; the $1$-cocycle condition corresponds in this case to the chain rule---exemplified by \eqref{eq_ex_entropic_chain_rule_deformed}---for certain restricted family of surjections encoded by an \emph{information structure} (a categorical object defined from a given collection of random variables).  This result does not require assumptions like the symmetry of $S_\alpha$ under permutations or its asymptotic behavior.  The construction is summarized in Section \ref{sec:info_structures}.

In Section \ref{sec:counting_functions}, we introduce a new module of coefficients $\sheaf G$ made of ``combinatorial'' functionals, and show that $1$-cocycles are in this case \emph{Fonten\'e-Ward generalized  multinomial coefficients}: given any sequence  $D=\{D_i\}_{i\geq 1}$ such that $D_1=1$, these coefficients are defined for any integers $k_1,...,k_s\in \Nn$ by
\begin{equation}\label{intro:fwcoeff}
\fwbinom{n}{k_1,...,k_s}_D:= \frac{[n]_D!}{[k_1]_D!\cdots[k_s]_D!},
\end{equation}
where $[n]_D!:=D_nD_{n-1}\cdots D_1$, $[0]_D!:=1$, and $n=\sum_{i=1}^s = k_i$.
Again, the $1$-cocycle condition implies all the multiplicative relations akin to \eqref{intro:recurrence-multinomial-3} for a given family of surjections encoded by the information structure. 

The generalized binomial coefficients were first introduced by Fontené in 1915 \cite{Fontene1915}, and later rediscovered by Ward \cite{Ward1936}, who developed a ``calculus of sequences'' analogue to the quantum calculus introduced by Jackson.  The multinomial case was already treated by  Gould \cite{Gould1969}.\footnote{Already Fonten\'e \cite{Fontene1915}, in 1915, noted that $\fwbinom{n}{k}_D :=\fwbinom{n}{k,n-k}_D$ verifies the \emph{additive} recurrence formula
\begin{equation}\label{eq:pascal}
\fwbinom{n}{k}_D -\fwbinom{n-1}{k}_D = \fwbinom{n-1}{k-1}_D \frac{D_n-D_{n-k}}{D_k},
\end{equation}
with boundary conditions $\fwbinom{n}{0}_D={\fwbinom nn}_D = 1$ for $n\geq 0$. Hence, for each sequence $D$,  the corresponding $D$-binomial coefficients are associated to certain  Pascal triangle defined in terms of $D$, see \cite[p.~25]{Gould1969}. } To our knowledge, three particular cases appear in the literature under their own name:
\begin{enumerate}
\item $D_n = n$ gives the usual multinomial coefficients.
\item $D_n = (q^n-1)/(q-1)$ gives the Gaussian $q$-multinomial coefficients, usually denoted $\qbinom{n}{k_1,...,k_s}$. See \cite{Vigneaux2019}.
\item When $D$ is the Fibonacci sequence and $s=2$, the expressions \eqref{intro:fwcoeff} are called Fibonomial coefficients.
\end{enumerate}

 The functions $f_D(\nu_1,\nu_2)= \fwbinom{\nu_1+\nu_2}{\nu_1,\nu_2}_D$ are the only solutions of  the functional equation
 \begin{equation}\label{eq:comb_FEITH_intro}
 \forall (\nu_0,\nu_1,\nu_2) \in \Nn^3 \sm \{(0,0,0)\}, \quad 
 \frac{f(\nu_0+\nu_1, \nu_2)}{f(\nu_0,\nu_2)}= \frac{f(\nu_1, \nu_0+\nu_2)}{f(\nu_1,\nu_0)}.
 \end{equation}
 This equation can be seen as a combinatorial version of the FEITH in view of the parallelism between  Proposition \ref{prop:func_eq} below and \cite[Prop.~3.10]{Vigneaux2019-thesis}.
%

We also prove that, for every $\alpha>0$, there is a generalized multinomial coefficient asymptotically related to the corresponding $\alpha$-entropy.  In fact, if $D_n^\alpha= \exp\{K(n^{\alpha-1}-1)\}$, for some $K\in \Rr$, then
\begin{equation}
\fwbinom{n}{p_1n,...,p_sn}_{D^\alpha}  = \exp\left\{n^\alpha \frac{K}{\alpha} S_\alpha(p_1,...,p_s)+o(n^\alpha)\right\}.
\end{equation}

Since the Fonten\'e-Ward multinomial coefficients  satisfy the same multiplicative relations as the usual multinomial coefficients, their logarithms (properly normalized) are connected in the limit $n\to \infty$ to the deformed chain rule \eqref{eq_ex_entropic_chain_rule_deformed}, as we already showed for the particular case $D_n = n$, which correspond to  $\alpha = 1$ and Shannon entropy. This gives an asymptotic correspondence between some of these new combinatorial $1$-cocycles and the old probabilistic $1$-cocycles, which is the subject of the last section.

\section{Information structures and their cohomology}\label{sec:info_structures}

An \emph{information structure} is a pair $(\Sinf, \sheaf E)$, made of a small category  $\Sinf$ and a functor  $\sheaf E:\Sinf\to \MeasSets_{surj}$, whose codomain is the category of measurable sets and measurable surjections between them.\footnote{In this article, boldface is associated to categories and caligraphic letters to functors. Given a functor $\sheaf F:\cat C\to \cat D$ and an object $X$ of $\cat C$, we denote by $\sheaf F_X$ the image of $X$ under $\sheaf F$ wherever is is possible---instead of the traditional notation $\sheaf F(X)$---to avoid excessive parentheses.} We denote by $(E_X, \salg B_X)$ the image of an object $X$ under $\sheaf E$.   The category $\Sinf$ is supposed 
\begin{enumerate}
\item to be a partially ordered set (poset): given any two objects $A$ and $B$ of $\Sinf$, there is at most one arrow from $A$ to $B$, and if $A\to B$ and $B\to A$, then $A=B$ (strict equality);
\item to have a terminal object $\Us$, and
\item to be ``conditionally cartesian'': for any diagram $X  \leftarrow Z \rightarrow Y$ in $\Sinf$, the categorical product $X\wedge Y$ exists.
\end{enumerate} 
In turn, the functor $\sheaf E$ is conservative (it does not turn nonidentity arrows into isomorphisms) and satisfies:
\begin{enumerate}
\item $E_\Us\cong \{\ast\}$,
\item for all $X\in \Ob \Sinf$, the $\sigma$-algebra $\salg B_X$ contains all the singletons $\{x\}\subset E_X$,  and
\item for every diagram 
 $
 \begin{tikzcd}
  X & X\wedge Y \ar[l,swap, "\pi"] \ar[r, "\sigma"] & Y 
 \end{tikzcd}
 $ in $\Sinf$, 
 the measurable map $$E_{X\wedge Y} \rightarrow E_X\times E_Y, \: z\mapsto (x(z),y(z)):= (\sheaf E\pi(z),\sheaf E \sigma(z))$$ is an injection.
\end{enumerate}

Information structures are combinatorial objects that accept a probabilistic interpretation, under which the objects of $\Sinf$, denoted $X,Y,Z,...$, represent random variables, and the functor $\sheaf E$ represents the possible outcomes of each variable. For any $X\in \Ob \Sinf$ and $A\in \salg B_X$, there is an \emph{event} $\{X\in A\}$. The arrows $\pi:X\to Y$ in $\Sinf$  correspond to the notion of refinement, which is implemented by the measurable map $\sheaf E\pi: \sheaf E_X \to \sheaf E_Y$: the event $\{Y\in A\}$ can also be defined in terms of $X$, as $\{X\in \sheaf E\pi^{-1}(A)\}$. The product $X\wedge Y$ represents the joint measurement of $X$ and $Y$, and the event $\{X\wedge Y = z\}$  gives an interpretation to the probabilistic notation $\{X= x(z), Y=y(z)\}$. 

There is an appropriate notion of morphism between information structures. This and some properties of the corresponding category are treated in \cite{Vigneaux2019-thesis}.

The information structure is finite if each set $E_X$ is finite; in this case, the algebra $\salg B_X$ is necessarily the atomic $\sigma$-algebra and can be omitted from the notation.  A treatment of the infinite case for gaussian random variables can be found in \cite{Vigneaux2019-thesis}.

Given an information structure $(\Sinf,   \sheaf E)$, one can define a presheaf (i.e. a contravariant functor) of monoids that maps $X\in \Ob \Sinf$ to the the set $\Smon_X  := \set{Y\in \Ob \Sinf}{X\to Y}$ equipped with the product  $(Y,Z) \mapsto YZ:= Y\wedge Z$;  an arrow  $X\to Y$ is mapped to the inclusion $\Smon_Y \hookrightarrow \Smon_X$. The associated presheaf of induced algebras   $X\mapsto  \Rr[\Smon_X]$ is denoted by $\Sring$. 

More generally, \emph{presheaves of sets} on $\Sinf$ are functors $\sheaf H: \Sinf^{op}\to \Sets$; a morphism $\phi:\sheaf H \to \sheaf K$ between presheaves is a natural transformation: a collection of mappings $\{\phi_X: \sheaf H_X \to \sheaf K_X\}$ such that, for every $\pi:X\to Y$ in $\Sinf$, the diagram
\begin{equation}\label{eq:naturality_nattrans}
\begin{tikzcd}
\sheaf H_Y \ar[r, "\phi_Y"] \ar[d, "\sheaf H \pi"] & \sheaf K_X \ar[d, "\sheaf K \pi"]  \\ 
\sheaf H_X \ar[r, "\phi_X"]  & \sheaf K_X
\end{tikzcd}
\end{equation}
in $\Sets$ commutes. One obtains in this way a category $\widehat S$, which is a basic example of a Grothendieck topos, see \cite{MacLane1994}. The product between two presheaves $\sheaf H$ and $\sheaf K$ is the presheaf that associates to each $X\in \Ob \Sinf$ the set $\sheaf H_X \times \sheaf K_X$ and to each arrow $\pi$ the map $\sheaf H\pi \times \sheaf K\pi$.

A \emph{ presheaf of $\Sring$-modules} is a presheaf of sets $\sheaf M$ together with a morphism $\phi:\sheaf A  \times \sheaf M \to \sheaf M$ such that, for every $X\in \Ob \Sinf$, the set $\sheaf M_X$ is an abelian group and $\phi_X: \sheaf A_X \times \sheaf M_X \to \sheaf M_X$ defines a structure of $\Sring_X$-module on $\sheaf M_X$.\footnote{After this paragraph, this mapping $\phi$ is always implicit: instead of $\phi(a,m)$, we write $a.m$.}  A morphism $\psi:\sheaf M\to \sheaf N$ between sheaves of $\Sring$-modules $(\sheaf M,\phi^{\sheaf M})$ and $(\sheaf N, \phi^{\sheaf N})$ is a morphism of presheaves $\psi:\sheaf M\to \sheaf N$ such that, for every $X\in \Ob \Sinf$, the mapping   $\psi_X$ is linear, and the diagram of presheaves
\begin{equation}\label{eq:naturality_modules}
\begin{tikzcd}
\Sring\times \sheaf M \ar[r, "\phi^{\sheaf M}"] \ar[d, "1\times \psi"] & \sheaf M \ar[d, "\psi"] \\
\Sring\times \sheaf N \ar[r,"\phi^{\sheaf N}"] & \sheaf N
\end{tikzcd}
\end{equation}
 commutes.  The set of $\Sring$-module morphisms from $\sheaf M$ to $\sheaf N$ is denoted by $\Hom_{\Sring}(\sheaf M,\sheaf N)$. Sheaves of ${\Sring}$-modules and its morphisms form the category $\Mod({\Sring})$.

\emph{Information cohomology} is a geometrical invariant associated to presheaves of $\Sring$-modules. It can be explicitly introduced as follows. 

First, for each $X\in \Ob \Sinf$, define  $\sheaf B_0(X)$ as the free $\sheaf A_X$-module  generated by the empty symbol $[\:]$, and $\sheaf B_n(X)$ as the  free $\sheaf A_X$-module  generated by $$\set{[X_1|\cdots|X_n]}{X_1,...,X_n\in \Smon_X}.$$ For every arrow $\pi:X\to Y$, there is an obvious inclusion $\sheaf B\pi: \sheaf B_i(Y) \hookrightarrow \sheaf B_i(X)$, for any $i\geq 0$; in fact, each $\sheaf B_i$ is a presheaf of $\sheaf A$-modules. 

Let $\Rr_{\Sinf}$ denote the constant sheaf, which associates to every object $X$ the vector space $\Rr$ with trivial $\Smon$-action and to every morphism the identity map. The presheaves $\{\sheaf B_i\}_{i\in \Nn}$ introduced above form a \emph{resolution} of $\Rr_\Sinf$, which means that there is a  diagram of presheaves
 \begin{equation}\label{eq:bar_resolution}
\begin{tikzcd}
0 
& \Rr_\Sinf \ar[l] 
& \sheaf B_0            \ar[l, "\epsilon", swap]  
& \sheaf B_1            \ar[l, "\partial_1", swap] 
& \sheaf B_2            \ar[l, "\partial_2", swap] 
& ...\ar[l, "\partial_3", swap] 
\end{tikzcd}
\end{equation}
such that $\im\partial_i = \ker \partial_{i-1}$ and $\im\partial_1 = \ker \epsilon$. These morphisms are defined on generators by the formulae 
$\epsilon([\,])= 1$, and 
\begin{multline}
\partial_{n} ([X_1|...|X_n]) = X_1[X_2|...|X_n] +\\ \sum_{k=1}^{n-1} (-1)^k [X_1|...|X_kX_{k+1}|...|X_n] +(-1)^n[X_1|...|X_{n-1}].
\end{multline}

Given any presheaf $\sheaf M$, we get a differential complex 
\begin{equation}\label{eq:diff_complex}
\begin{tikzcd}[column sep=small]
0 \ar[r]
& \Hom_{\Sring}(\Rr_\Sinf,   \sheaf  M) \ar[r] 
&  C^0(\sheaf M)          \ar[r, "\delta^0"]  
&  C^1(\sheaf M)             \ar[r, "\delta^1"]
&  C^2(\sheaf M)           \ar[r, "\delta_2"]
& ...
\end{tikzcd},
\end{equation}
\normalsize
where $C^n( \sheaf M)$ denotes $\Hom_{\Sring}(\sheaf   B_n,   \sheaf M)$ and each morphism $\delta^i:  C^i(\sheaf M) \to  C^{i+1}(\sheaf M)$ is given by the formula $\delta^i (\phi) := \phi \circ \partial_{i+1}:\sheaf B_{i+1}\to \sheaf M$. In general, this complex  is not exact, but  $\delta^{i+1}\circ \delta^i = 0$ still holds for every $i\in \Nn$. 

The \emph{information cohomology} of $\Sinf$ with coefficients in $\sheaf M$, denoted $H^\bullet(\Sinf, \sheaf M)$, is the cohomology of the differential complex \eqref{eq:diff_complex}, this is, 
 \begin{equation}
  H^0(\Sinf,\sheaf M):=\ker \delta^0 \quad \text{and} \quad H^n(\Sinf,\sheaf M) := \ker \delta^n/\im\delta^{n-1} \text{ when } n\in \Nn^*.
  \end{equation} 
 The elements of $C^n(\sheaf M)$ are called $n$-cochains: they are $n$-cocycles when they belong to $Z^n(\sheaf M):= \ker \delta^n$, and $n$-coboundaries when they belong to $\delta^n C^{n-1}$. We omit the superindex of $\delta$ if it is clear from context. Every $n$-coboundary is an $n$-cocycle, but the converse is not true. An $n$-cochain $\phi$ is by definition a collection $\{\phi_X:\sheaf B_n(X) \to \sheaf M_X\}$ of $\sheaf A_X$-equivariant mappings, see \eqref{eq:naturality_modules}. Therefore, it is enough to determine the image $\phi_X([X_1|\cdots|X_n])$ of each generator $[X_1|\cdots|X_n]$ of $\sheaf B_n(X)$; to simplify notation, we write  $\phi_X[X_1|\cdots|X_n]$. The naturality with respect to $X$---this is, the commutativity of \eqref{eq:naturality_nattrans}---translates into the following condition: for every arrow $\pi:X\to Y$ in $\Sinf$,
\begin{equation}\label{eq:joint_locality}
\phi_X[X_1|\cdots|X_n] = \sheaf M\pi(\phi_Y[X_1|\cdots|X_n])
\end{equation}
whenever $\{X_1,...,X_n\}\subset \Smon_Y \hookrightarrow \Smon_X$. Remark that any variable $Y$ that refines $X_1$, $X_2$,..., $X_n$, also refines their product $X_1\cdots X_n$; thus \eqref{eq:joint_locality} is equivalent to
\begin{equation}\label{eq:joint_locality2}
\phi_X[X_1|\cdots|X_n] = \sheaf M\rho(\phi_{X_1\cdots X_n}[X_1|\cdots|X_n])
\end{equation}
where $\rho$ is the arrow $X\to X_1\cdots X_n$. According to this equation, $\phi_X[X_1|\cdots|X_n] $ only depends on its ``localization'' at $X_1\cdots X_n$; consequently, we refer to \eqref{eq:joint_locality} or \eqref{eq:joint_locality2} as \emph{joint locality}.

\begin{remark}
The category of $\Sring$-modules is abelian and has enough injectives. Therefore, one can introduce cohomological $\partial$-functors in the sense of \cite{Grothendieck1957} (see also \cite{Weibel1994}). The functors $\{\Ext^i(\Rr_\Sinf,-)\}_{i\geq 0}$ are the right derived functors of $\Hom(\Rr_\Sinf,-)$. Information cohomology with coefficients $\sheaf M$ can be defined as $H^\bullet(\Sinf,\sheaf M) :=\Ext^\bullet(\Rr, \sheaf M)$. These cohomology groups are naturally isomorphic to those introduced above, because  it can be proved that each $\sheaf B_i$ is a projective object in  $\Mod(\Sring)$. For details, see \cite[Sec.~2.4]{Vigneaux2019-thesis}.
\end{remark}

We introduce now a concrete example related to probabilistic functionals studied in \cite{Baudot2015} and \cite{Vigneaux2019-thesis}. We assume from now on that $(\Sinf,\sheaf E)$ is a finite information structure.

Let $\sheaf P$ be the functor that associates to any $X\in \Ob\Sinf$ the set of probabilities 
\begin{equation}\label{eq:sheaf_probas}
\sheaf P_X := \bigset{p:E_X\to [0,1]}{\sum_{x\in E_X} p(x) = 1}, 
\end{equation}
and to each morphism $\pi:X\to Y $, the mapping $\sheaf P\pi:\sheaf P_X \to \sheaf P_Y$ given by 
\begin{equation}\label{eq:marginalization_probas}
\sheaf P\pi(p) (y) := \sum_{x\in \sheaf E\pi^{-1}(x)} p(x), 
\end{equation}
called \emph{marginalization}.  When $\pi$ is clear from context, we write $Y_*p$ instead of $\sheaf P\pi(p)$.

Given any probability $p\in \sheaf P_X$, an arrow $\pi:X\to Y$, and $y\in E_Y$ such that $Y_*p(y)\neq 0$, the \emph{conditional probability} $p|_{Y=y}:E_X\to [0,1]$ is given by
\begin{equation}
p|_{Y=y}(x) := \begin{cases}\frac{p(x)}{Y_*p(y)} & \text{if } x\in \sheaf E\pi^{-1}(y) \\ 0 & \text{otherwise}\end{cases}.
\end{equation}

Let $\sheaf F$ be the presheaf that associates to each $X\in \Ob \Sinf$ be the real vector space of measurable functions $f:\sheaf P_X \to \Rr$,\footnote{The set $\sheaf P_X$ can be naturally identified with the standard simplex $\Delta^{|E_X|-1}$, which equipped with its Borel $\sigma$-algebra is a measurable space.} and to each arrow $\pi:X\to Y$ in $\Sinf$, the map given by precomposition with the corresponding marginalization: $\sheaf F\pi (f) =f \circ \sheaf P \pi$. 

For any $\alpha>0$, we define an action of $\Smon_X$ on $\sheaf F_X$ as follows: for each $Y\in \Smon_X$, $f\in   \sheaf  F_X$ and $p\in   \sheaf P_X$,  
\begin{equation}
(Y.f)(p) = \sum_{\substack{y\in E_Y\\ Y_*P(y)\neq 0}} (Y_*p(y))^\alpha \phi(p|_{Y=y}).
\end{equation}
Extended linearly, this turns $\sheaf F_X$ into an $\Sring_X$-module. Since the action is natural,  we obtain an $\Sring$-module denoted $\sheaf F_\alpha$. See Proposition 3.1 and 3.2 in \cite{Vigneaux2019-thesis}.

We call $H^\bullet(\Sinf, \sheaf F_\alpha)$ \emph{probabilistic information cohomology}. Probabilistic $0$-cochains $\phi\in  C^0(\sheaf F_\alpha)$ are given by a collection of functions $\{\phi_X[\:] \in \sheaf F_\alpha(X)\}_X\in \Ob{\Sinf}$ that by joint locality must be constant: $\phi_X[\:](P_X)= \phi_\Us[\:](\Us_*P_X)=\phi_\Us[\:](\delta_{\ast})\in \Rr$. It is not difficult to see that $Z^0(\sheaf F_1) = C^0(\sheaf F_1)$, hence  $H^0(\Sinf,\sheaf F_1)\cong \Rr$, whereas $H^0(\Sinf,\sheaf F_\alpha)=0$ for every $\alpha\neq 1$. In turn, any probabilistic $1$-cochain $\phi\in C^1(\sheaf F_\alpha)$ satisfies $\phi_X[Z](p)=\phi_Z[Z](Z_*p)$ by joint locality \eqref{eq:joint_locality}. In fact, the collection of measurable functions $\{\phi[Z]:\sheaf P_Z\to \Rr\}_{Z\in\Ob\Sinf}$ defines the $1$-cochain. Hence each $\alpha$-entropy determines a $1$-cochain $S_\alpha \in C^1(\sheaf F_\beta)$, for any $\beta>0$, through the formulae
\begin{equation}
\forall p \in \sheaf P_X, \quad S_1[X](p) := -\sum_{x\in E_X} p(x) \ln p(x)
\end{equation}
and 
\begin{equation}\label{eq:tsallis-entropy-def}
\forall p \in \sheaf P_X, \quad S_\alpha[X](p)= \frac{1}{1-\alpha} \left(\sum_{x\in E_X} p(x)^\alpha -1\right),
\end{equation}
when  $\alpha \in (0,\infty)\sm\{1\}$. 

Moreover, $S_\alpha$ is a $1$-cocycle of type $\alpha$ i.e. an element of $Z^1(\sheaf F_\alpha)$. The cocycle condition $\delta S_\alpha = 0$ means that, for every $X\in \Ob\Sinf$ and every $Y,Z\in \Smon_X$, the equation
\begin{equation}
0 = (Y.S_\alpha)_X[Z]-(S_\alpha)_X[YZ]+(S_\alpha)_X[Y]
\end{equation}
holds, and this corresponds exactly to the $\alpha$-chain rule, cf. \eqref{eq_ex_entropic_chain_rule_deformed}.\footnote{Information theorists would write $H(Y,Z) = H(Y) + H(Z|Y)$ in the case of Shannon entropy.} Conversely, the equation $0=Y.\phi[Z] - \phi[YZ] + \phi[Y]$  (where marginalizations are implicit) has in general a unique solution, provided that the product $YZ$ is \emph{nondegenerate}, which means that $E_{YZ}$ is ``close'' to $E_Y\times E_Z$ in a sense made precise by \cite[Def.~3.12]{Vigneaux2019-thesis} or Definition \ref{def:non-degenerate-combinatorial} in Section \ref{sec:combinatorial_cohomology}. 

\begin{prop}[\protect{\cite[Prop.~3.13]{Vigneaux2019-thesis}, see also \cite{Baudot2015}}]\label{H_for_nondegenerate_products}
Let $(\Sinf,\sheaf E)$ be a finite information structure and $X$, $Y$ two different variables in $\Ob{\Sinf}$ such that $XY\in \Ob{\Sinf}$. Let $\phi$ be a $1$-cocycle of type $\alpha$. If $XY$ is nondegenerate, there exists $\lambda\in \Rr$ such that
$$\phi[X]=\lambda S_\alpha[X], \quad \phi[Y]=\lambda S_\alpha[Y], \quad \phi[XY]=\lambda S_\alpha[XY].$$
\end{prop}

The following result specifies the global number of free constants. It applies to any poset  $\Sinf$ with bounded height:\footnote{The height of a poset is the length of the longest chain of morphisms $a_1\to a_2 \to ... \to a_n$, where no arrow equals an identity map.} we say in this case that $(\Sinf,\sheaf E)$ is bounded.

\begin{prop}[\protect{\cite[Thm.~3.14]{Vigneaux2019-thesis}}]\label{H1-non-degenerate}
Let $(\Sinf,\sheaf E)$ be a bounded, finite information structure. Denote by $\Sinf^\ast$ the full subcategory of $\Sinf$ generated by $\Ob\cat S \sm\{\mathbf 1\}$. Suppose that every minimal object can be factored as a nondegenerate product.  Then, 
\begin{equation}
H^1(\Sinf,\sheaf F_1) = \prod_{[\cat C] \in \pi_0(\Sinf^*) } \Rr \cdot S_1^{\cat C}
\end{equation}
and, when $\alpha\neq 1$,
\begin{equation}
H^1(\Sinf,\sheaf F_\alpha) = \left(\prod_{[\cat C]\in \pi_0(\Sinf^*) } \Rr \cdot S_\alpha^{\cat C}\right)/\Rr\cdot S_\alpha 
\end{equation}
In the  formulae above, $\cat C$ represents a connected component of $\Sinf^\ast$, and 
$$S_\alpha^{\cat C}[X] = \begin{cases}S_\alpha[X] & \text{if } X \in \Ob\cat C \\
0 & \text{if } X \notin \Ob \cat C
\end{cases}$$
\end{prop}

\section{Counting functions}\label{sec:counting_functions}
 
Let $(\Sinf,\sheaf E)$ be a finite  information structure, and $\sheaf C:\Sinf \to \Sets$ a functor that associates to each object $X$ the set 
\begin{equation}
\sheaf C_X=\bigset{\nu:E_X\to \Nn}{\sum_{x\in E_X} \nu(x) > 0},
\end{equation}
and to each arrow $\pi:X\to Y$, associated to a surjection $\sheaf E\pi:E_X\to E_Y$, the map $\sheaf C\pi:\sheaf C_X \to \sheaf C_Y$ that verifies $\sheaf C\pi(\nu)(y)=\sum_{x\in\sheaf E\pi^{-1}(y)}\nu(x)$. To simplify notation, we  write $Y_* \nu$ instead of $\sheaf C\pi(\nu)$, whenever $\pi$ is clear from context. The elements of $\sheaf C_X$ are called \keyt{counting functions}.\index{counting function} For $\nu_X\in \sheaf C_X$, we define its \keyt{support}  as $\set{x\in E_X}{\nu_X(x)\neq 0}$, and its \keyt{magnitude} as the quantity $\norm{v}:= \sum_{x\in X} \nu(x)$.  

For any subset $A$ of $X$, there is a restriction  
\begin{equation}
\nu|_{A} (x) := \begin{cases} \nu(x) & \text{if } x\in A \\ 
0 & \text{otherwise}\end{cases}.
\end{equation}
When $\norm{\nu|_{A}}>0$, we call $\nu|_{A}$ the \keyt{restricted counting given $A\subset X$}.\index{counting function!restricted}
Given an arrow $\pi:X\to Y$, the notation  $\nu|_{Y=y}$ stands for $\nu|_{\sheaf \pi^{-1}(y)}$. Remark that $\nu_\emptyset= 0$ and $\norm{\nu|_{Y=y}}= Y_*\nu(y)$.

Consider now the multiplicative abelian group $\sheaf G_X$, whose elements are $\Rr_+^*$-valued measurable functions defined on $\sheaf C_X$. By $\Rr_+^*$ we mean $\set{x\in \Rr}{x> 0}$. (The multiplicative notation is convenient, because multinomial coefficients appear directly as cocycles.) The group $\sheaf G_X$ becomes a real vector space if we define $(r.g)(\nu):= (g(\nu))^r$, for each $g\in \sheaf G_X$ and each $r\in \Rr$. \footnote{In principle this is a right action, but this is immaterial because $\Rr$ is commutative.} For each $Y\in \Smon_X$ and each $g\in \sheaf G_X$, set
\begin{equation}
(Y.g)(\nu) := \prod_{\substack{y\in E_Y \\ Y_*\nu(y)\neq 0}} g(\nu|_{Y=y}).
\end{equation}
Finally, define $(aY).g := a.(Y.g) = Y.(a.g)$. As a consequence of the following proposition, these formulae give an homomorphism $\rho_X:\Sring_X \to \End(\sheaf G_X)$, that turns $\sheaf G_X$ into an $\Sring_X$-module.

\begin{prop}
Given variables $Y,Z\in \Smon_ X$ and $g\in \sheaf G_X$,
\begin{equation}
ZY.g = Z.(Y.g).
\end{equation}
\end{prop}
\begin{proof}
Set $W$ equal to $ZY:= Z\wedge Y$. Since in $\Sinf$ we have the commutative diagram
$$
\begin{tikzcd}[column sep=huge, row sep=large]
	& X \ar[ld, swap, "\pi_{YX}"] \ar[rd, "\pi_{ZX}"] \ar[d, "\pi_{WX}"] 
 		&  \\ 
 Y 
 	&  W \ar[l, "\pi_{YW}"] \ar[r, swap, "\pi_{ZW}" ]
 		& Z
 	\end{tikzcd}
 	$$
 	we obtain the following commutative diagram of sets
 	$$
\begin{tikzcd}[column sep=huge, row sep=large]
	& E(X) \ar[ld, swap, "\pi_{YX}"] \ar[rd, "\pi_{ZX}"] \ar[d, "\pi_{WX}"] 
 		&  \\ 
 E(Y) &  E(W) \ar[l, "\pi_{YW}"] \ar[r, swap, "\pi_{ZW}"] \ar[d, hookrightarrow, "\iota"]
 		& E(Z) \\
 		& E(X)\times E(Y) \ar[ul, "\pi_1"] \ar[ur, swap,"\pi_2"] &
 	\end{tikzcd}
 	$$
 	where the upper triangle is explained by the functoriality of $\sheaf E$ (to simplify notation, we write $\pi$ instead of $\sheaf E\pi$) and the lower one by the universal property of products in $\Sets$. The mapping $\iota$ is an injection by  definition of an information structure. 
 	
 	Note that
 	\begin{align}
 	Z.(Y.g)(\nu) & = \prod_{\substack{z\in E_Z \\ Z_*\nu(z)\neq 0}} (Y.g)(\nu|_{Z=z}) \label{product1}\\
&= \prod_{\substack{z\in E_Z \\ Z_*\nu(z)\neq 0}} \prod_{\substack{y\in E_Y \\ Y_*\nu|_{Z=z}(y)\neq 0}} g((\nu|_{Z=z})|_{Y=y}).\label{product2}
\end{align}
From the definition of conditioning, we deduce that $(\nu|_{Z=z})|_{Y=y} = \nu|_{\{Z=z\}\cap\{Y=y\}}= \nu|_{A(y,z)}$, where we have set $$A(y,z):=\pi_{YX}^{-1}(y)\cap\pi_{ZX}^{-1}(z) = \pi_{WX}^{-1}\iota^{-1}(\pi_1^{-1}(y) \cap \pi_2^{-1}(x)).$$ If $(y,z) \not\in \im \iota$, $A(y,z)$ is empty, so $\nu|_{A(y,z)}=0$, as well as $\norm{\nu|_{A(y,z)}} = Y_*\nu|_{Z=z}(y)=0$. Therefore, the product in \eqref{product2} can be restricted to pairs $(y,z) \in \im \iota$, and the condition $Y_*\nu|_{Z=z}(y)=\norm{\nu|_{A(y,z)}}\neq 0$ translates into $W_* \nu(\iota^{-1}(y,z))=\norm{\nu|_{A(y,z)}}\neq 0$. Since there is a bijection $ E_W \cong \im \iota$, upon relabeling we obtain the desired equality.   
\end{proof}

To any arrow $\pi:X\to Y$, we associate the map $\sheaf G\pi:\sheaf G_Y\to \sheaf G_X$ such that $\sheaf G\pi(g) = g \circ \sheaf C(\pi)$. Then $\sheaf G:\Sinf \to \Sets$ is a contravariant functor. In fact, it is a presheaf of $\Sring$-modules:  it is not difficult to prove that the commmutivity of \eqref{eq:naturality_modules} holds, cf. \cite[Prop.~3.2]{Vigneaux2019-thesis}.

\section{Combinatorial information cohomology}\label{sec:combinatorial_cohomology}
 In this section, we compute the information cohomology of $\Sinf$ with coefficients in $\sheaf G$, which we call \keyt{combinatorial information cohomology}. See Section \ref{sec:info_structures}.
 
 The elements of $C^n(\sheaf G) :=\Hom_\Sring (\sheaf B_n, \sheaf G)$ are called \emph{combinatorial $n$-cochains}. The coboundary of $\psi\in C^n(\sheaf G)$ is the $(n+1)$-cochain $\delta \psi: \sheaf B_{n+1} \to \sheaf G$ defined on the generators of $\sheaf B_{n+1}$ by
\begin{multline}\label{eq:general_coboundary}
\delta \psi[X_1|...|X_{n+1}] =\\ (X_1.\psi[X_2|...|X_{n+1}]) \left(\prod_{k=1}^{n}  (\psi[X_1|...|X_kX_{k+1}|...|X_n])^{(-1)^k}\right) {\psi[X_1|...|X_{n}]}^{(-1)^{n+1}},
\end{multline}
because we are using multiplicative notation for $\sheaf G$. 
A combinatorial  $n$-cocycle is an element $\psi$ in $C^n(\Sinf, \sheaf G)$ that verifies $\delta \psi=1$; the submodule of all $n$-cocycles is denoted by $Z^n(\sheaf G)$. The image under $\delta$ of $C^{n-1}$ forms another submodule of $C^n(\sheaf G)$, denoted $\delta C^{n-1}(\sheaf G)$; its elements are called combinatorial $n$-coboundaries.

\subsection{Computation of $H^0$}

The $0$-cochains are given by a collection of functions $\{\psi_X\}_{X\in \Ob \Sinf}$ (the image of the generator $[\:]$ under $\psi$ over each $X$). Joint locality implies that, for every $X\in \Ob  \Sinf$, $\psi_X(\nu) = \psi_{\mathbf 1}(\mathbf 1_*\nu_X) = \psi_{\mathbf 1}(\norm{\nu_X})$. Hence, $0$-cochains are in one-to-one correspondence with measurable functions of the magnitude, $\Psi:=\psi_{\mathbf 1} : \Nn^* \to \Rr_+$.

A $0$-cocycle $\psi$ must verify, for each $Y$ coarser than $X$, the equation $(\delta \psi)_X[Y]= (Y.\psi_X)(\psi_X)^{-1}=1$, which is equivalent to
\begin{equation}\label{general_cocycle_0}
 \Psi (\norm{\nu_X})= \prod_{\substack{y\in Y \\ Y_*\nu(y)\neq 0}}  \Psi (\norm{\nu|_{Y=y}}).
\end{equation}
Whenever $|Y|\geq 2$, this means in particular that
\begin{equation}
\Psi(x+y) = \Psi(x)\Psi(y)
\end{equation}
for every $x,y \in \Nn^*$. Setting $a:=\Psi(1)>0$, one easily concludes by recurrence that $\Psi(n) = a^n = \exp(n\ln(a))$.   The function $\Psi(x) = \exp(kx)$, for arbitrary $k\in \Rr$, is a general solution of \eqref{general_cocycle_0}, because $\norm{\nu_X}= \sum_{\substack{y\in Y \\ Y_*\nu(y)\neq 0}} \norm{\nu|_{Y=y}}$. We have proved the following proposition.

\begin{prop}
Let $\Exp \in \Hom_{\Sring}(\ast,\sheaf G)$ be the section defined by 
\begin{equation*}
\Exp_X:\sheaf C_X \to \Rr^*_+, \quad \nu\mapsto \exp(\norm\nu).
\end{equation*}
 Then $H^0(\Sinf, \sheaf G) = \langle \Exp \rangle_\Rr$.
\end{prop}

\subsection{Computation of $H^1$}

 For any $1$-cochain $\psi$, we set $\psi[Z]:= \psi_Z[Z] = \psi_X[Z]$, the last equality being valid for any $X$ such that $X\to Z$ by joint locality.

In order to compute the $1$-cocycle, we prove first an auxiliary result.

\begin{lem}\label{prop:1_cocycles_G}
Let $\psi\in Z^1(\sheaf G)$. For every $X\in \Ob \Sinf$,
if $\nu\in \sheaf C_X$ verifies $\nu = \nu|_{X=x_0}$ for some $x_0\in E_X$, then $\psi[X](\nu)= 1$. 
\end{lem} 
In particular, $\psi[\mathbf 1] \equiv 1$. 
\begin{proof}
The cocycle condition implies in particular that $\psi[XX]=(X.\psi[X])\psi[X]$, this is 
\begin{equation}\label{product_for_1_cocycle}
1=\prod_{\substack{x\in E_X \\ \nu(x)\neq 0}} \psi[X](\nu|_{X=x}) = \psi[X](\nu|_{X=x_0}).
\end{equation} 
\end{proof}

The following result will be essential for the characterization of all the $1$-cocycles. It is the combinatorial analogue of \cite[Prop.~3.10]{Vigneaux2019-thesis}, where a variant of  the fundamental equation of information theory \eqref{eq:FEITH-def} appears. Consequently, \eqref{fundamental_equation_discrete} and \eqref{eq:func_eq_1} can be seen as combinatorial generalizations of this functional equation.  

\begin{prop}[Combinatorial FEITH]\label{prop:func_eq}
Let $f_1,f_2 :\Nn\sm \{(0,0)\}\to \Rr_+$ be two unknown functions. The functions $f_1,f_2$ satisfy the conditions
\begin{enumerate}
\item\label{prop:func_eq_cond1} for $i\in \{1,2\}$, for every $n \in \Nn^*$, $f(n,0)=f(0,n)=1$.
\item\label{prop:func_eq_cond2}  for every $\nu_0, \nu_1, \nu_2\in \Nn$ such that $\nu_0+\nu_1+\nu_2\neq 0$,
\begin{equation}\label{fundamental_equation_discrete}
f_1(\nu_0+\nu_2,\nu_1)f_2(\nu_0,\nu_2)=f_2(\nu_0+\nu_1,\nu_2)f_1(\nu_0,\nu_1).
\end{equation}
\end{enumerate} 
if, and only if, there exists a sequence of numbers $D= \{D_i\}_{i\geq 1} \subset \Rr_+$,  such that $D_1=1$, and
\begin{equation}
f(\nu_1,\nu_2) = \frac{[\nu_1+\nu_2]_D!}{[\nu_1]_D![\nu_2]_D!},
\end{equation}
 where $[n]_D! = D_n D_{n-1}\cdots D_1$ whenever $n>0$, and $[0]_D!=1$.
\end{prop}
 \begin{proof}
 Setting $\nu_0=0$, we conclude first that $f_1(\nu_2,\nu_1)=f_2(\nu_1,\nu_2)$. Define $f(x,y):= f_1(x,y)$; it satisfies the equation 
 \begin{equation}\label{eq:func_eq_1}
 \frac{f(\nu_0+\nu_1, \nu_2)}{f(\nu_0,\nu_2)}= \frac{f(\nu_1, \nu_0+\nu_2)}{f(\nu_1,\nu_0)}.
 \end{equation}
 for any $\nu_0, \nu_1, \nu_2\in \Nn$ such that $\nu_0+\nu_1+\nu_2\neq 0$.  In particular, if $\nu_0=t>0$, and $\nu_1=\nu_2=s>0$, 
 \begin{equation}
 \frac{f(t+s,s)}{f(s,t+s)}=\frac{f(t,s)}{f(s,t)}.
 \end{equation}
 Thus,
 for any $n>1$,
 \begin{equation}
 \frac{f(n,1)}{f(1,n)}=  \frac{f(n-1,1)}{f(1,n-1)}= \cdots =  \frac{f(1,1)}{f(1,1)}=1.
 \end{equation}
 Let $D_{n+1}$ be the common value of $f(n,1)$ and $f(1,n)$. 
 From Equation \eqref{eq:func_eq_1}, setting $\nu_0=n$, $\nu_1=1$, and $\nu_2=k$, we can obtain a recurrence formula for $f(n+1,k)$:
 \begin{equation}
 f(n+1, k)= \frac{D_{n+k+1}}{D_{n+1}} f(n,k).
 \end{equation}
By repeated application of this recurrence, we conclude that
 \begin{equation}
 f(n,k)= \frac{D_{n+k}}{D_n}\cdot\frac{D_{n+k-1}}{D_{n-1}}\cdots \frac{D_{k+1}}{D_1} f(0,k).
 \end{equation}
 Remark that $D_1=f(0,1)=1$, and $f(0,k)=1$ (Lemma \ref{prop:1_cocycles_G}). Therefore, $f$ can be rewritten as
 \begin{equation}\label{solution_f}
 f(\nu_1,\nu_2) = \frac{[\nu_1+\nu_2]_D!}{[\nu_1]_D![\nu_2]_D!}.
\end{equation}  
This formula still make sense when $\nu_1=0$ or $\nu_2=0$. Conversely, for any sequence $D= \{D_i\}_{i\geq 1}$, with $D_1=1$, the assignment $f_1=f_2 =f$ satisfies \eqref{eq:func_eq_1}, and thus represents the most general solution.
\end{proof}

\begin{ex}
Let $(\Sinf, \sheaf E)$ be an information structure defined as follows: $\Sinf$ is the poset  represented by the graph
$$
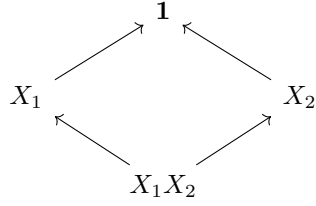
\begin{tikzcd}
& \Us & \\
X_1 \ar[ur] & & X_2 \ar[ul] \\
& X_1X_2 \ar[ul] \ar[ur]& 
\end{tikzcd}
$$ 
and $E$ is the functor defined at the level of objects by $E(X_1) =\{x_{\{1\}}, x_{\{0,2\}}\}$, $E(X_2)=\{x_{\{2\}}, x_{\{0,1\}}\}$, and $E(X_1X_2) =\{x_{\{1\}},x_{\{2\}},x_{\{3\}}\}$; for each arrow $\pi:X\to Y$, the map $\pi_*:E(X)\to E(Y)$ sends $x_I\to x_J$ iff $I\subset J$. 

For this structure, the cocycle condition give the equations
\begin{align}
\psi[X_1X_2](\nu_0,\nu_1,\nu_2) &= \psi[X_2](\nu_0+\nu_1,\nu_2) \psi[X_1](\nu_0,\nu_1)\psi[X_1](\nu_2,0), \label{ex:012_cocycle_eq_1}\\
\psi[X_2X_1](\nu_0,\nu_1,\nu_2) &= \psi[X_1](\nu_0+\nu_2,\nu_1) \psi[X_2](\nu_0,\nu_2)\psi[X_2](\nu_1,0).
\end{align}
Since $X=X_1X_2=X_2X_1$,
\begin{equation}
\psi[X_2](\nu_0+\nu_1,\nu_2) \psi[X_1](\nu_0,\nu_1) = \psi[X_1](\nu_0+\nu_2,\nu_1) \psi[X_2](\nu_0,\nu_2)
\end{equation}
where we have taken into account that $\psi[X_1](\nu_2,0)=\psi[X_2](\nu_1,0)=0$. This is exactly Equation \eqref{fundamental_equation_discrete}, and the condition \eqref{prop:func_eq_cond1} in the statement is also met, therefore
\begin{equation}
\psi[X_1](\nu_0,\nu_1) = \psi[X_2](\nu_0,\nu_1) = \fwbinom{\nu_0+\nu_1}{\nu_0,\nu_1}_D
\end{equation}
for some sequence $D$. From \eqref{ex:012_cocycle_eq_1}, we conclude that
\begin{equation}
\psi[X](\nu_0,\nu_1,\nu_2) = \fwbinom{\nu_0+\nu_1+\nu_2}{\nu_0,\nu_1,\nu_2}_D := \frac{[\nu_0+\nu_1+\nu_2]_D!}{[\nu_0]_D! [\nu_1]_D! [\nu_2]_D!}.
\end{equation}
\end{ex}

\begin{defi} Given any sequence $D=\{D_i\}_{i\geq 1}$ verifying $D_1=1$ (called \keyt{admissible sequence}\index{sequence!admissible}), the corresponding \keyt{Fonten\'e-Ward multinomial coefficient}\index{multinomial coefficient!Fontené-Ward} is the $1$-cochain given by
\begin{equation}
\forall \nu \in \sheaf C(X), \quad W_D[X](\nu) = \frac{[\norm{\nu}]_D !}{\prod_{x\in E_X} [\nu(x)]_D!}.
\end{equation}
\end{defi}

To characterize the $1$-cocycles $\psi$ in the general case, we introduce a notion of nondegenete product analogous to \cite[Def.~3.12]{Vigneaux2019-thesis}.\footnote{Both notions coincide when $\sheaf Q$ in \cite{Vigneaux2019-thesis} is the functor $\sheaf P$ introduced by \eqref{eq:sheaf_probas}.} Its definition is better understood reading the proof of Proposition \ref{H_for_nondegenerate_products_combinatorial}. The idea is to determine the function $\psi[XY]$, for given variables $X$ and $Y$, applying the same kind of reasoning used in the previous example. One obtains first the recursive formulae \eqref{comb:recurrence_X} and \eqref{comb:recurrence_Y} for the functions $\psi[X]$ and $\psi[Y]$: these are based on a chosen total order of the sets $E_X$ and $E_Y$, and the steps of the recursion are coded by a path in $\Zz^2$. Both recursive formulae  are a simplification of the symmetric equation \eqref{comb:eq_XY} for particular laws $\tilde \nu$ given by the Condition \ref{combinatorial-ng:lifting} in Definition \ref{def:non-degenerate-combinatorial} that make one of the factors trivial. These recursive formulae involve a term where $\psi[X]$ and $\psi[Y]$ have only two nonzero arguments, and one recovers ``locally'' the combinatorial FEITH of Proposition \ref{prop:func_eq}: the fact that three different integers are involved in this equation is ensured by Condition \ref{combinatorial-ng:2-dim-cell-nondegenerate} in Definition \ref{def:non-degenerate-combinatorial}.

\vspace{5pt}
\begin{defi}\label{def:non-degenerate-combinatorial}
Let $X$ and $Y$ be two objects of $\Sinf$, such that $|E_X|=k$ and $|E_Y|=l$. Let $\iota$ be the inclusion $E_{XY}\hookrightarrow E_X\times E_Y$. We call the product $XY$ \keyt{nondegenerate} if there exist enumerations  $\{x_1,...,x_k\}$ of $E_X$ and  $\{y_1,...,y_l\}$ of $E_Y$, and a North-East (NE) lattice path\footnote{A North-East (NE) lattice path  on $\Zz^2$ is a sequence of points  $(\gamma_i)_{i=1}^m\subset \Zz^2$ such that $ \gamma_{i+1} -\gamma_i\in \{(1,0),(0,1)\}$ for every $i\in \{1,...,m-1\}.$}   $(\gamma_i)_{i=1}^m$  on $\Zz^2$  going from $(1,1)$ to $(k,l)$ such that
\begin{enumerate}  
\item\label{combinatorial-ng:lifting} If $\gamma_i=(a,b)$ and $\gamma_{i+1} -\gamma_i = (1,0)$, we ask that for every counting function  $\nu\in \sheaf C_X$ such that $\supp \nu \subset \set{x_i}{ a \leq i \leq k} $, there exists a counting function $\tilde \nu \in \sheaf C_{XY}$ whose support is contained in
\small
$$
\iota^{-1}(\{(x_a,y_{b+1})\}\cup \set{(x_i,y_b)}{a+1 \leq i \leq k})\cup \iota^{-1}(\{(x_a,y_{b})\}\cup \set{(x_i,y_{b+1})}{a+1 \leq i \leq k})$$ 
\normalsize
and such that $\nu = X_*\tilde \nu$. Remark that, for such values of $XY$, the value of the $Y$-component completely determine the $X$-component.

Analogously, if $\gamma_{i+1} -\gamma_i = (0,1)$, we ask that every counting function $\nu\in \sheaf C_Y$ such that  $\supp \nu \subset\set{y_i}{b\leq i \leq l}$, there exists a counting function $\tilde \nu \in \sheaf C_{XY}$ whose support is contained in
\small
$$
\iota^{-1}(\{(x_{a+1},y_{b})\}\cup \set{(x_a,y_j)}{b+1 \leq j \leq l})\cup \iota^{-1}(\{(x_a,y_{b})\}\cup \set{(x_{a+1},y_j)}{b+1 \leq j \leq k})$$
\normalsize
and such that $\nu = Y_*\tilde \nu$. 

\item\label{combinatorial-ng:2-dim-cell-nondegenerate} For each $\gamma_i=(a,b)$, the set 
$$\iota^{-1}\set{(x_i,y_j)}{ a\leq i \leq {a+1} \text{ and } b \leq j \leq {b+1}} $$
contains at least three different elements.
\end{enumerate}
 \end{defi}
 \vspace{5pt}
 \begin{prop}\label{H_for_nondegenerate_products_combinatorial}
Let $(\Sinf,\sheaf E)$ be a finite information structure and $X$, $Y$ two different variables in $\Ob{\Sinf}$ such that $XY\in \Ob{\Sinf}$. Let $\psi$ be a combinatorial $1$-cocycle i.e. an element of $Z^1(\Sinf,\sheaf G)$. If $XY$ is nondegenerate, there exists an admissible sequence $D$, such that
$$\psi[X]=W_D[X], \quad \psi[Y]=W_D[Y], \quad \psi[XY]=W_D[XY].$$
\end{prop}
\begin{proof} 
Since $\psi$ is a $1$-cocycle, it satisfies the two equations derived from \eqref{eq:general_coboundary} 
\begin{align}
Y.\psi[X]\psi[Y] & =  \psi[XY], \label{comb:proof_rec_1} \\
X.\psi[Y]\psi[X] & = \psi[XY] . \label{comb:proof_rec_2}
\end{align}
and therefore the symmetric equation
\begin{equation}\label{comb:eq_XY}
(X.\psi[Y])\psi[X]=(Y.\psi[X])\psi[Y].
\end{equation}
For any counting function  $\nu$, we write  $$\left(\begin{array}{cccc}
s & t & u & \ldots \\
p & q & r & \ldots
\end{array}\right)$$ if $\nu(s) = p$, $\nu(t) = q$, $\nu(u) = r$, etc. and the images of the unwritten parts are zero.

Fix an order $(x_1,...,x_k)$ and $(y_1,...,y_l)$ that satisfies the definition of nondegenerate product, and let $\{\gamma_{i}\}_{i=0}^m$ be the corresponding path. If $\gamma_i=(a,b)$ and $\gamma_{i+1}-\gamma_i=(1,0)$, we are going to show that the following recursive formula holds:
\begin{equation}\label{comb:recurrence_X}
\psi[X]\left(\begin{array}{ccc}
x_{a} & \ldots & x_k \\
\mu_{a} & \ldots & \mu_k
\end{array}\right)=   \psi[X]\left(\begin{array}{ccc}
x_{a+1} & \ldots & x_k \\
\mu_{a+1}& \ldots & \mu_k
\end{array}\right) \psi[X]\left(\begin{array}{cc}
x_{a} & x_{a+1} \\
\mu_{a} & \norm\mu - \mu_a
\end{array}\right).
\end{equation}
Analogously, if $\gamma_i=(a,b)$ and $\gamma_{i+1}-\gamma_i=(0,1)$,
\begin{equation}\label{comb:recurrence_Y}
\psi[Y]\left(\begin{array}{ccc}
y_{b} & \ldots & y_l \\
\nu_b & \ldots & \nu_l
\end{array}\right)=  \psi[Y]\left(\begin{array}{ccc}
y_{b+1} & \ldots & y_l \\
\nu_{b+1}& \ldots & \nu_{l}
\end{array}\right)  \psi[Y]\left(\begin{array}{cc}
y_b & y_{b+1} \\
\nu_b & \norm\nu - \nu_b
\end{array}\right).
\end{equation}

Suppose that $\gamma_i=(a,b)$ and $\gamma_{i+1}-\gamma_i=(1,0)$.  Let $$\mu=\left(\begin{array}{cccc}
x_a & \ldots & x_k  \\
\mu_a & \ldots & \mu_k 
\end{array}\right)$$ be a counting function in $\sheaf C_X$. By Definition \ref{def:non-degenerate-combinatorial}-\ref{combinatorial-ng:lifting} above, we know that $\mu$ has a preimage under marginalization $\tilde \mu$, whose support is such that $(X.\psi[Y])(\tilde \mu)=1$, cf. Lemma \ref{prop:1_cocycles_G}.   Equation \eqref{comb:eq_XY} then reads 
\begin{equation}\label{auxiliary1_proof_recurrence}
 \psi[X]\left(\begin{array}{ccc}
x_{a+1} & \ldots & x_k \\
\mu_{a+1}& \ldots & \mu_k
\end{array}\right)\psi[Y]\circ \tau\left(\begin{array}{cc}
y_b & y_{b+1} \\
\norm\mu-\mu_a &  \mu_a
\end{array}\right) = 
\psi[X]\left(\begin{array}{ccc}
x_{a} & \ldots & x_k \\
\mu_{a} & \ldots & \mu_k
\end{array}\right),
\end{equation}
where $\tau$ is the identity or the transposition of the nontrivial arguments of $\mu$. In any case, setting  $\mu_{a+1} = \norm \mu-\mu_a$ and $\mu_{a+2}=\ldots = \mu_k = 0$, we conclude that
\begin{equation}
\psi[X]\left(\begin{array}{cc}
x_{a} & x_{a+1} \\
\mu_{a} & \norm\mu-\mu_a
\end{array}\right)=\psi[Y]\circ\tau \left(\begin{array}{cc}
y_b & y_{b+1} \\
\norm\mu-\mu_a &  \mu_a
\end{array}\right),
\end{equation}
which combined with \eqref{auxiliary1_proof_recurrence} implies \eqref{comb:recurrence_X}. The identity \eqref{comb:recurrence_Y} can be obtained analogously.

To determine 
$$f_a(n_1,n_2):= \psi[X]\left(\begin{array}{cc}
x_{a} & x_{a+1} \\
n_1 & n_2
\end{array}\right) \quad \text{and} \quad g_b(n_1,n_2):= \psi[Y]\left(\begin{array}{cc}
y_b & y_{b+1} \\
n_1 & n_2
\end{array}\right),$$
for $(n_1,n_2)\in \Nn^2\sm \{(0,0)\}$, consider the three different elements $w_1,w_2,w_3$ in $E_{XY}\subset E_X\times E_Y$ given by the property \ref{combinatorial-ng:2-dim-cell-nondegenerate} of a nondegenerate product. The symmetric equation \eqref{comb:eq_XY} evaluated on $\nu_1\delta_{w_1} + \nu_2\delta_{w_2} + \nu_3 \delta_{w_3}\in \sheaf C_{XY}$ gives  the equation that appears in Proposition \ref{prop:func_eq}, which implies that $f_a(n_1,n_2) = g_b(n_1,n_2) = \fwbinom{n_1+n_2}{n_1,n_2}_D$ for certain admissible sequence $D$ (the eventual permutations of the arguments in the unknowns become irrelevant, because the solution is symmetric).  

When considering $\gamma_{i+1}$, one finds the functions $f_a$ and $g_{b+1}$, or the functions $f_{a+1}$ and $g_b$, since $\gamma_{i+1}-\gamma_i$ is either $(1,0)$ or $(0,1)$. This ensures that the admissible sequence $D$ that appears for each $\gamma_i$ is always the same, as proved in Lemma \ref{lemma:unicity_D}. The recurrence relations \eqref{comb:recurrence_X} and \eqref{comb:recurrence_Y} then imply  the desired result. 
\end{proof}

\begin{lem}\label{lemma:unicity_D} Let $D, D'$ be two admissible sequences. If for all $n_1,n_2\in \Nn^2$
\begin{equation}
\fwbinom{n_1+n_2}{n_1,n_2}_D = \fwbinom{n_1+n_2}{n_1,n_2}_{D'}
\end{equation}
then $D=D'$.
\end{lem}
\begin{proof}
Just remark that 
\begin{equation}
\fwbinom{n}{1}_D =\fwbinom{n}{1,n-1}_D = [n]_D!
\end{equation}
so we have $[n]_D! = [n]_{D'}!$ for all $n\in \Nn$.
\end{proof}

As in the continuous case, the number of admissible sequences that appear in the computation of the $1$-cocycles $Z^1(\Sinf, \sheaf G)$ depends on the number of connected components of $\Sinf^*$, that is $\Sinf$ deprived of its final element.
In addition, a choice of $0$-cochain $\psi$ induces globally a Fonten\'e-Ward coefficient $\delta \psi$  for a unique admissible sequence $D_g$. 
%
 Therefore, $Z^1(\sheaf G)$ and $\delta C^0(\sheaf G)$ are both infinite dimensional. If $\Sinf^*$ is connected, the quotient is trivial; otherwise it is infinite: $|\pi_0(\Sinf^*)|-1$ admissible sequences remain arbitrary. This is the combinatorial version of  Proposition \ref{H1-non-degenerate}.

\section{Asymptotic relation with probabilistic information cohomology}\label{sec:symptotics_combinatorial_cohomology}

\begin{prop}\label{prop:asymptotic-relation-n-cocycles}
Let $\psi$ be a combinatorial $n$-cocycle. Suppose that, for every $X_1,...,X_n \in \Ob \Sinf$ such that $X_i\cdots X_n \in \Ob \Sinf$, there exists a measurable function $$\phi[X_1|...|X_n]: \sheaf P_{X_1\cdots, X_n} \to \Rr$$ with the following property: for every sequence of counting functions $\{\nu_n\}_{n\geq 1}\subset \sheaf C_{X_1\cdots X_n}$ such that
\begin{enumerate}
\item $\norm {\nu_n} \to \infty$, and
\item for every $z\in E_{X_1\cdots X_n}$, $\nu_n(z)/\norm {\nu_n} \to p(z)$ as $ n\to \infty$
\end{enumerate} 
the asymptotic formula
$$\psi[X_1|...|X_n](\nu_n) = \exp({\norm {\nu_n}}^\alpha \phi[X_1|...|X_n](p) + o(\norm {\nu_n}^\alpha))$$
holds. Then $\phi$ is a $n$-cocycle of type $\alpha$, i.e. $\phi\in Z^n(\Sinf, \sheaf F_\alpha(\sheaf P))$.
\end{prop}
\begin{proof}
To simplify notation, we assume that $n=1$; the proof is still valid in the general case. We must show that, for every $p\in \sheaf P_{XY}$,
$$\phi[XY](p) = (X.\phi[Y])(p)+\phi[X](X_*p).$$
Let $\{\nu_n\}_{n\geq 1}$ be a sequence of counting functions such that $\norm {\nu_n} \to \infty$ and, for every $z\in E_{XY}$, $\nu_n(z)/\norm{\nu_n} \to p(z)$. A sequence like this always exists: just consider a rational approximation of the values of $p$ with common denominator. 

Since $\psi$ is a $1$-cocycle, $\psi[XY] = (X.\psi[Y])\psi[X]$. Evaluate it at $\nu_n$, take the logarithm and divide by $\norm{\nu_n}^\alpha$ in order to  obtain
\begin{equation}\label{eq:discrete-1cocycle-additive-form-of-n}
\frac{\ln \psi[XY](\nu_n)}{\norm{\nu_n}^\alpha} = \sum_{\substack{x\in E_X \\ X_*\nu_n(x) \neq 0}} \frac{\ln \psi[Y](\nu_n |_{X=x})}{\norm{\nu_n}^\alpha}  + \frac{\ln \psi[X](\nu_n)}{\norm{\nu_n}^\alpha}.
\end{equation}
Recall that, for any counting function $\nu$, $\norm{\nu|_{X=x}}= X_*\nu(x)$. Hence, 
\begin{equation}
\frac{\ln \psi[Y](\nu_n |_{X=x})}{\norm{\nu_n}^\alpha} = \frac{\ln \psi[Y](\nu_n |_{X=x})}{\norm{\nu_n |_{X=x}}^\alpha} \frac{(X_*\nu_n(x))^\alpha}{\norm{\nu_n}^\alpha}.
\end{equation}
Plug this in \eqref{eq:discrete-1cocycle-additive-form-of-n} and take the limit as $n$ goes to infinity to conclude.
\end{proof}

We discuss now some important examples:
\begin{enumerate}
\item The exponential $\Exp^k: \nu \to \exp(k \norm \nu)$ is the a combinatorial $0$-cocycle, and it corresponds to the constant  $k$ seen as a probabilistic $0$-cocycle.

\item As we explained in Section \ref{sec:intro},
\begin{equation}
{n \choose p_1n,...,p_sn} = \exp(n S_1(p_1,...,p_s)+o(n))
\end{equation}
The multinomial coefficient is a combinatorial $1$-cocycle and $S_1$ defines an element of $Z^1(\sheaf F_1)$.

\item Whereas the previous examples are not  surprising, Proposition \ref{prop:asymptotic-relation-n-cocycles} hints at new objects that are connected to the generalized $\alpha$-entropies and have gone unnoticed until now. For example, the $q$-multinomial coefficients are connected asymptotically to the $2$-entropy (quadratic entropy),
\begin{equation}
\qbinom{n}{p_1n,...,p_sn} = \exp(n^2 \frac{\ln q}{2} S_2(p_1,...,p_s)+o(n^2)),
\end{equation}
see \cite[Prop.~2]{Vigneaux2019}.  These coefficients have a combinatorial interpretation: when $q$ is a prime power and $k_1,...,k_s$ are integers such that $\sum_{i=1}^s k_i = n$, the coefficient  $\qbinom{n}{k_1,...,k_s}$ counts the number of flags of vector spaces $V_1\subset V_2 \subset ... \subset V_n = \Ff_q^n$ such that $\dim V_i = \sum_{j=1}^i k_j$ (here $\mathbb{F}_q$ denotes the finite field of order $q$).  In particular, the $q$-binomial coefficient $\qbinom nk \equiv \qbinom{n}{k,n-k}$ counts vector subspaces of dimension $k$ in $\Ff_q^n$.

In \cite{Vigneaux2019}, we push this parallel between $S_1$ and $S_2$  much further: we introduce a probabilistic model that generates vector spaces and study its concentration properties, in order to obtain a  generalization of the Asymptotic Equipartition Property that involves the quadratic entropy.

\end{enumerate}

 It is quite natural to ask if, for any $\alpha > 0$,  there exists a sequence $D^\alpha=\{D^\alpha_i\}_{i\geq 1}$ asymptotically related to the entropy $S_\alpha$ through Proposition \ref{prop:asymptotic-relation-n-cocycles}. The answer turns out to be yes.
\begin{prop}
Consider any $\alpha\in \Rr_+^* \sm \{1\}$. If $D_n^\alpha= \exp\{K({n^{\alpha-1}-1})\}$, for any $K\in \Rr$, then
$$\fwbinom{n}{p_1n,...,p_sn}_{D^\alpha}  = \exp\left\{n^\alpha \frac{K}{\alpha} S_\alpha(p_1,...,p_s)+o(n^\alpha)\right\}.$$
\end{prop}
\begin{proof}
Remark that $[n]_D! := \exp\{ K(\sum_{i=1}^n i^{\alpha-1}  - n)\}$. 

Suppose first that $\alpha > 1$. In this case, $x\mapsto x^{\alpha-1}$ is strictly increasing and 
\begin{equation}
\int_{0}^n x^{\alpha-1} \d x = \frac{n^\alpha}\alpha < \sum_{i=1}^n i^{\alpha-1} < \int_{1}^{n+1} x^{\alpha-1} \d x  = \frac{(n+1)^\alpha}{\alpha} - \frac{1}{\alpha}.
\end{equation}
Hence, if $K>0$,
\begin{equation}\label{eq:bounds_sum_ialpha}
\exp\left\{K \left(\frac{n^\alpha}\alpha - n\right) \right\} < [n]_D ! <  \exp\left\{K \left(\frac{(n+1)^\alpha}{\alpha} - \frac{1}{\alpha}-n\right) \right\}.
\end{equation}
This directly implies that 
\begin{equation}
\frac{[n]_D!}{[n_1]_D! \cdots [n_s]_D!}  < \exp\left\{ K \left(\frac{(n+1)^\alpha}{\alpha} - \frac{1}{\alpha}-n -  \sum_{i=1}^s \left( \frac{n_i^\alpha} \alpha - n_i \right) \right) \right\} \nonumber\\
 \label{eq:proof_asymptotics_fontene_upperbound}
\end{equation}
as well as
\begin{equation}
\frac{[n]_D!}{[n_1]_D! \cdots [n_s]_D!} > \exp\left\{ K \left(\frac{n^\alpha}{\alpha}-n -  \sum_{i=1}^s \left( \frac{(n_i+1)^\alpha} \alpha - n_i \right) \right) \right\} \nonumber\\
 \label{eq:proof_asymptotics_fontene_lowerbound}
\end{equation}
from which the conclusion follows. 

If $K<0$, the inequalities \eqref{eq:bounds_sum_ialpha}, \eqref{eq:proof_asymptotics_fontene_upperbound} and \eqref{eq:proof_asymptotics_fontene_lowerbound} must be reversed, but the result is the same. Similarly, when $0<\alpha<1$ the argument remains valid making the necessary modifications: all inequalities are reversed, since $x\mapsto x^{\alpha-1}$ is strictly decreasing. 
\end{proof}

It is not known if these or similar coefficients related to $S_\alpha$, for $\alpha \in \Rr_+^*\sm\{1,2\}$, have a combinatorial or statistical interpretation.
%
%
%

\bibliographystyle{abbrv}
\bibliography{Bibliography}

\end{document}